\documentclass{amsart}
\usepackage{amsmath, amssymb, amsthm, amscd, amsfonts, eucal, hyperref}
\usepackage{enumerate}
\usepackage[T5, T1]{fontenc}
\usepackage[all]{xy}

\newtheorem{theorem}{Theorem}[section]
\newtheorem{proposition}[theorem]{Proposition}
\newtheorem{lemma}[theorem]{Lemma}
\newtheorem {corollary}[theorem]{Corollary}
\theoremstyle {definition}
\newtheorem {definition}[theorem]{Definition}
\newtheorem {example}[theorem]{Example}

\newtheorem {notation}[theorem]{Notation}
\theoremstyle {remark}
\newtheorem{remark}[theorem]{Remark}

\def\Spec{\operatorname{Spec}}

\def\Ker{\operatorname{Ker}}

\def\codim{\operatorname{codim}}

\newcommand{\fm}{\ensuremath{\mathfrak m}}
\newcommand{\fn}{\ensuremath{\mathfrak n}}

\newcommand{\fp}{\ensuremath{\mathfrak p}}
\newcommand{\fq}{\ensuremath{\mathfrak q}}

\newcommand{\cO}{\ensuremath{\mathcal O}}

\newcommand{\cX}{\ensuremath{\mathcal X}}

\newcommand{\bC}{\ensuremath{\mathbb C}}
\newcommand{\bR}{\ensuremath{\mathbb R}}

\begin{document}

\title[Fibers of flat morphisms and Weierstrass preparation theorem]{Fibers of flat morphisms and Weierstrass preparation theorem}

\author{\fontencoding{T5}\selectfont \DJ o\`an Trung C\uhorn{}\`\ohorn ng}
\address{\fontencoding{T5}\selectfont \DJ o\`an Trung C\uhorn{}\`\ohorn ng. {\it Current address:} Vietnam Institute for Advanced Study in Mathematics, Ta Quang Buu Building, 01 Dai Co Viet, Hai Ba Trung, Hanoi, Vietnam. {\it Permanent address:} Institute of Mathematics, 18 Hoang Quoc Viet, 10307 Hanoi, Vietnam.}
\email{doantc@gmail.com}

\dedicatory{\fontencoding{T5}\selectfont Dedicated to Professor Nguy\~\ecircumflex
n T\d \uhorn\
 C\uhorn{}\`\ohorn ng}

\thanks{This research is funded by Vietnam National Foundation for Science and Technology Development (NAFOSTED) under grant number 101.01-2012.05.}


\begin{abstract} 
We characterize flat extensions of commutative rings satisfying the Weierstrass preparation theorem. Using this characterization we prove a variant of the Weierstrass preparation theorem for rings of functions on a normal curve over a complete local domain of dimension one. This generalizes recent works of Harbater, Hartmann and Krashen with a different method of proof.
\end{abstract}

\maketitle

\setcounter{section}{-1}
\section{Introduction}
\label{sect0}

The Weierstrass preparation theorem is an important result in the theory of several complex variables. Its main idea is that the local  behavior of a holomorphic function is similar to the behavior of a polynomial \cite{gh}. There are several ways to generalize this theorem for different purposes in other contexts. The algebraic form of the Weierstrass preparation theorem has important applications in algebra and algebraic number theory. For example, the usage of the algebraic form in the study of quadratic forms is well-known (cf. \cite{cdlr}).

Recently Harbater-Hartmann-Krashen have generalized the algebraic form of the Weierstrass preparation theorem to functions on curves over a complete discrete valuation ring. Using essentially techniques of patching, they showed that a fraction of elements in certain completions of the ring of regular functions on the curve factors as product of a rational function and an invertible power series \cite{hh, hhk1, hhk2}. This result was then applied to study quadratic forms and central simple algebras over the related fields of functions. A significant application is to show that a field extension of transcendence degree one over the $p$-adic numbers has u-invariant $8$. This is one of three (recent) proofs for this long-standing problem (see \cite{cop}, \cite{hhk1}, \cite{dl}, \cite{ps}).

The aim of the current paper is to generalize further the theorem of Harbater-Hartmann-Krashen on the Weierstrass preparation theorem with a different approach through ring theory and algebraic geometry. Weierstrass preparation is a kind of factorization assertion for elements of a ring extension. If the extension under consideration satisfies the Weierstrass preparation theorem (in the sense of Definition \ref{11} in Section \ref{sect1}) then all the fibers have dimension zero. This observation is our starting point. In the works of Harbater, Hartmann and Krashen, the authors consider a normal curve over a complete discrete valuation ring and study the functions that are rational on a subset of an irreducible component of the closed fiber. In our approach, we consider functions that are regular on an arbitrary subset of the closed fiber without restriction to a component. A careful analysis on the fibers of flat extensions allows us to treat the branching phenomenon when passing from functions on a component to the general case. It also gives another proof for the generalized Weierstrass preparation theorem of Harbater-Hartmann-Krashen. It is remarkable that we mainly work on the level of rings rather than on the fields.

About the structure of the paper, we introduce in Section \ref{sect1} a notion of Weierstrass homomorphism and recall an application of the Weierstrass preparation theorem in the study of the u-invariant of a field. In Section \ref{sect2} we consider flat homomorphisms which are Weierstrass. The main result of the section is Theorem \ref{23} presenting a characterization of such a Weierstrass flat homomorphism via properties of the fibers and the residue fields. This characterization is used in Section \ref{sect3} to generalize the Weierstrass preparation theorem for functions on curves over a complete local domain of dimension one (see Theorems \ref{33}, \ref{35}, \ref{36}). A consequence of the characterization theorem in Section \ref{sect2} is that over a local ring, a henselian Weierstrass extension is (up to isomorphism) intermediate between the Henselization and the completion of the ring. So the Weierstrass property of the three extensions should relate to each other. This is discussed in Section \ref{sect4}.

Throughout this paper all rings are commutative with a unit.


\section{Weierstrass preparation theorem and u-invariant}
\label{sect1}

The Weierstrass preparation theorem has important applications in the algebraic theory of quadratic forms. To motivate the works in this paper, we recall here a well-known application of the theorem in the study of the u-invariant of a field. For recent works in this direction and other applications of the Weierstrass preparation theorem, we mention Harbater-Hartmann-Krashen \cite{hhk1, hhk2}.

\begin{definition}\label{11}
Let $A, B$ be commutative rings and $\varphi: A\rightarrow B$ be a ring homomorphism. We say that the homomorphism $\varphi$ is Weierstrass if for any $b\in B$, there is an element $a\in A$ such that $b=\varphi(a)b^\prime$ for some unit $b^\prime\in B^\times$. If in addition $A\subseteq B$ is a subring, we say that $B$ is a Weierstrass extension of $A$.
\end{definition}

\begin{example}\label{12} 
We have some examples of Weierstrass homomorphisms.
\item(a) Any surjective ring homomorphism is Weierstrass.
\item(b) Let $S$ be a multiplicatively closed subset of a commutative ring $A$, then the localization map $A\rightarrow S^{-1}A$ is Weierstrass.
\item(c) Let $X$ be a normal curve over a field and $x\in X$ be a point. Let $A=\cO_{X, x}$ and $\hat A$ be the completion of $A$ with respect to the maximal ideal. As $A$ and $\hat A$ are discrete valuation rings, the completion map $A\rightarrow \hat A$ is Weierstrass. 
\item(d) Let $\alpha: A\rightarrow B$, $\beta: B\rightarrow C$ be homomorphisms of commutative rings. If both $\alpha$ and $\beta$ are Weierstrass then the composition $\beta\circ\alpha$ is Weierstrass. Conversely, if the composition $\beta\circ\alpha$ is Weierstrass then so is $\beta$. If in addition, $\beta$ is injective and $\beta(B^\times)=C^\times$, then $\alpha$ is Weierstrass too.
\end{example}

\begin{remark}\label{13}
The Weierstrass property is not a local property. If $f: A\rightarrow B$ is Weierstrass then so is the induced homomorphism $\varphi_P: A\rightarrow B_P$ for any prime ideal $P$ of the ring $B$. However, the converse is not true. For example, consider the canonical map $\varphi: \bR[X]\rightarrow \bC[X]$. Clearly $\varphi$ is not Weierstrass since there is no expression such as $X+i=a(X)b(X)$, for some $a(X)\in \bR[X]$ and invertible $b(X)\in \bC[X]$. Let $P=(X-c)$ be a maximal ideal of $\bC[X]$, $c\in \bC$. Let $\varphi_P: \bR[X]\rightarrow \bC[X]_{(X-c)}$ be the induced homomorphism. An element of $\bC[X]_{(X-c)}$ is of the form $(X-c)^nf(X)$ for some $n\geq 0$ and $f(X)\in \bC[X]_{(X-c)}$ invertible. If $c\in \bR$ then obviously $\varphi_P$ is Weierstrass. If $c\not\in \bR$ then $(X-c)(X-\bar c)$ is a real polynomial and $X-\bar c$ is invertible in $\bC[X]_{(X-c)}$, where $\bar c$ is the complex conjugate of $c$. So $(X-c)^nf(X)=(X-c)^n(X-\bar c)^n\frac{f(X)}{(X-\bar c)^n}$ and $\varphi_P$ is Weierstrass in this case.
\end{remark}

Let $k$ be a field. The u-invariant of $k$, denoted by $u(k)$, is the maximal dimension of all anisotropic quadratic forms over $k$. If the maximum does not exist, we set $u(k)=\infty$. For example, $u(\bR)=\infty$ and $u(k)=0$ if $k$ is algebraically closed. The u-invariant $u(k)$ is a very important numerical invariant of $k$ but it is hard to track information about this invariant (see \cite{lty, ps}). For instance, the question on determining all possible values of this invariant is one of the major open problems in the theory of quadratic forms. The idea of using the Weierstrass preparation theorem to study the u-invariants is well-known and has been worked out by several authors (see, for example, \cite{cdlr}). Before going to this, we need first the following consequence of Hensel's Lemma.

\begin{lemma}\label{14}
Let $A\subseteq B$ be an extension of commutative rings. Let $I$ be an ideal of $A$. Suppose $B$ is henselian relative to $IB$ and $A/I\simeq B/IB$. Let $n>0$ be a positive integer which is invertible in $B/IB$. For a unit $b\in B^\times$, there are $a\in A^\times$ and $\alpha\in B$ such that $b=a\alpha^n$.
\end{lemma}
\begin{proof}
Since $A/I\simeq B/IB$, there is a unit $a\in A^\times$ such that $a\equiv b \pmod{IB}$. Let $F(X)=aX^n-b\in B[X]$. Then $F^\prime(1)=na$ is invertible in $B/IB$. Moreover,
$$F(1)=a-b \equiv 0 \pmod{IB}.$$
Since $B$ is henselian relative to $IB$, there is uniquely an element $\alpha\in B$ such that $F(\alpha)= a
\alpha^n-b=0$ and $\alpha\equiv 1\pmod{IB}$. 
\end{proof}

For a domain $A$, we denote $K_A$ for the field of fractions of $A$. 

\begin{theorem}\label{15}
Let $A$ be a domain and let $A\subseteq B$ be a Weierstrass extension. Let $I$ be an ideal of $A$ such that $B$ is henselian relative to $IB$ and $A/I\simeq B/IB$. Suppose $2$ is invertible in $B/IB$. Then $B$ is a domain and 
$$u(K_B)\leq u(K_A).$$
\end{theorem}
\begin{proof} Obviously $B$ is a domain as the extension $A\subseteq B$ is Weierstrass. We first show that $K_B$ is a compositum of $K_A$ and $B$, that is, $K_A.B=K_B$. It suffices to show that $K_A.B$ is a field since $K_A.B\subseteq K_B$. Let $\frac{a_1}{a_2}b$ be an element of $K_A.B$, where $a_1, a_2\in A$ and $b\in B$, $a_1, a_2, b\not=0$. By the Weierstrass property, $b=ab^\prime$ for some $a\in A$ and some unit $b^\prime\in B^\times$. Thus $\frac{a_1}{a_2}b=\frac{aa_1}{a_2}b^\prime$ which is invertible in $K_A.B$. This shows that $K_A.B$ is a field and $K_A.B=K_B$.

Set $n=u(K_B)$. Let $$\varphi(x)=\beta_1x_1^2+\ldots + \beta_nx_n^2,$$
be an anisotropic quadratic form with $\beta_1, \ldots, \beta_n\in K_B$. Since $K_A.B=K_B$, as discussed above we have
$\beta_i=\frac{a_i}{s_i}b_i,$ for $a_i, s_i\in A$, $s_i\not=0$ and $b_i\in B^\times$ a unit, $i=1, \ldots, n$. Note that $A/I\simeq B/IB$ and $B$ is henselian relative to $IB$. Using Lemma \ref{14}, we obtain an expression $b_i=c_i\alpha_i^2$ for some $c_i\in A^\times, \alpha_i\in B^\times$, $i=1, \ldots n$. So
$$\varphi(x)=\frac{a_1c_1}{s_1}(\alpha_1 x_1)^2+\ldots+ \frac{a_nc_n}{s_n}(\alpha_n x_n)^2,$$
which defines an anisotropic quadratic form over $K_A$. Therefore $u(K_A)\geq u(K_B)$.
\end{proof}

\begin{remark}\label{16} (i) In the application of Theorem \ref{15}, usually the field $K_B$ is henselian or even complete. There are many cases the u-invariant of $K_B$ is computable. The well-known theorem of Springer is an example. Theorem \ref{15} thus gives a bound for the u-invariant of $K_A$.\smallskip

\noindent (ii) For any $n$ invertible in $A/I$, the proof of Theorem \ref{15} works for any homogeneous form of degree $n$ instead of just quadratic forms and over rings rather than fields. In particular, we get a result similar to Theorem \ref{15} for higher degree homogeneous forms.

\end{remark}


\section{Weierstrass flat homomorphisms}
\label{sect2}

This section is devoted to a careful study of Weierstrass flat homomorphisms. The main result is Theorem \ref{23} in which a characterization of Weierstrass flat homomorphisms in the local case is given. We begin with a couple of lemmas.

\begin{lemma}\label{21}
Let $\varphi: A\rightarrow B$ be a homomorphism of commutative rings. Assume $B$ has a unique maximal ideal $\fm_B$. The following statements are equivalent:

(a) The homomorphism $\varphi$ is Weierstrass.

(b) Any ideal $I$ of $B$ has a set of generators in the image of $\varphi$, that means, $I=\varphi(\varphi^{-1}(I))B$.

(c) Any principal ideal of $B$ has a set of generators in the image of $\varphi$.
\end{lemma}
\begin{proof}
Replace $A$ by $\varphi(A)$ we can assume $A\subseteq B$. The implication $(a)\Rightarrow (b)$ and the equivalence $(b)\Leftrightarrow(c)$ are obvious. We will prove $(b)\Rightarrow (a)$.

Take an element $b\in B$ and set $J=(bB)\cap A$. Since $bB=JB$, $b=\sum_{i=1}^n a_ib_i$ for some $a_1, \ldots, a_n\in J$, $b_1, \ldots, b_n\in B$. On the other hand, as elements of $J\subseteq bB$, $a_i=c_ib$ for some $c_i\in B$, $i=1, \ldots, n$. So $b(1-\sum_{i=1}^nb_ic_i)=0$. If there is an element $c_i$ invertible, $b=c_i^{-1}a_i$ is the required expression. Otherwise, $c_1, \ldots, c_n$ are all in $\fm_B$ and hence $1-\sum_{i=1}^nb_ic_i$ is invertible. This implies that $b=0$.
\end{proof}

The second lemma should be known by experts. We include it with proof in this section for the sake of completeness.

\begin{lemma}\label{22}
Let $A, B$ be Noetherian local rings  with maximal ideals $\fm_A, \fm_B$ respectively. Assume there is a faithfully flat homomorphism $\varphi: A \rightarrow B$ such that $\sqrt{\fm_AB}=\fm_B$. Let $M$ be an $A$-module of finite length. Then $M\otimes_AB$ is a $B$-module of finite length and $\ell_B(M\otimes_AB)=\ell_A(M)\ell_B(B/\fm_AB)$.
\end{lemma}
\begin{proof}
Take a chain of submodules of $M$, say $0=M_0\subset M_1\subset \ldots \subset M_n=M$, where $M_i/M_{i-1}\simeq A/\fm_A$, $i=1, \ldots, n$, and $n=\ell_A(M)$. The flatness of the homomorphism $\varphi$ gives rise to a chain of submodules of $M\otimes_AB$, 
$$0=M_0\otimes_AB\subset M_1\otimes_AB\subset \ldots \subset M_n\otimes_AB.$$
Note that $(M_i\otimes_AB)/(M_{i-1}\otimes_AB)\simeq (M_i/M_{i-1})\otimes_AB\simeq (A/\fm_A)\otimes_AB\simeq B/\fm_AB$. So $\ell_B(M\otimes_AB)=\ell_A(M)\ell_B(B/\fm_AB)$ which is finite.
\end{proof}

It is usually hard to use directly the definition of Weierstrass homomorphism in order to show a homomorphism being Weierstrass. Lemma \ref{21} provides a criterion which is still hard to be used. In the next we present a criterion for a flat homomorphism to be Weierstrass by means of conditions on the prime ideals and associated residue fields. This criterion will be shown to be very useful in several situations in the next sections. 

For a prime ideal $\fp$ of a ring $R$, we denote by $k_R(\fp):=R_\fp/\fp R_\fp$ the residue field.

\begin{theorem}\label{23} 
Let $\varphi: A\rightarrow B$ be a flat homomorphism of Noetherian rings, where $B$ has a unique maximal ideal $\fm_B$. The homomorphism $\varphi$ is Weierstrass if and only if any prime ideal $P\in \Spec(B)$ satisfies

(a) $P=\varphi(\varphi^{-1}(P))B$, that means, $P$ has a set of generators in the image of $\varphi$;

(b) Either $B_P$ is a principal ideal ring or the canonical inclusion $k_A(\varphi^{-1}(P))\hookrightarrow k_B(P)$ is an isomorphism (in this case we will identify the two fields).
\end{theorem}

\begin{proof}
Replace $A$ by $\varphi(A)$ we can assume $A\subseteq B$.\medskip

\noindent{\it Necessary condition}: Assume $A\subseteq B$ is a Weierstrass extension. Let $P$ be a prime ideal of $B$ and let $\fp=P\cap A$. Lemma \ref{21} gives $P=\fp B$. It remains to show that $k_A(\fp)=k_B(P)$ if $B_P$ is not a principal ideal ring.

Since $P$ is the only prime ideal in the fiber at $\fp$, we have $B\otimes_AA_\fp\simeq B_P$. So the map $A_\fp\rightarrow B_P$ is flat by base change and it is obviously Weierstrass. The proof is thus reduced to the case of local rings $(A, \fm_A), (B, \fm_B)$ such that $B$ is not a principal ideal ring, $\fm_AB=\fm_B$ and $P=\fm_B$.

We are going to show that $A/\fm_A=B/\fm_B$. Now let $\bar f\in B/\fm_B$. Let $x_1, \ldots, x_n\in \fm_A$ be a minimal set of generators of $\fm_A$. Applying Lemma \ref{22} to the $A$-module $\fm_A/\fm_A^2$ we obtain
 $$n:=\dim_{A/\fm_A} (\fm_A/\fm_A^2)=\ell_A(\fm_A/\fm_A^2)=\ell_B(\fm_B/\fm_B^2)=\dim_{B/\fm_B}(\fm_B/\fm_B^2).$$ 
Consequently, $x_1, \ldots, x_n$ is a minimal set of generators of $\fm_B$. Since $B$ is not a principal ideal ring, $\fm_B$ is not a principal ideal and particularly, $n>1$.

Set $b=x_1+fx_2\in \fm_B$. The assumption of Weierstrass extension gives rise to an expression $b=ab^\prime$ for some $a\in \fm_A$ and $b^\prime \in B^\times$. Write $a=\sum_{i=1}^na_ix_i$. In the $B/\fm_B$-vector space $\fm_B/\fm_B^2$ we have
$$\bar x_1+\bar f\bar x_2=\overline{ab^\prime}=\sum_{i=1}^n\overline{a_ib^\prime}\bar x_i,$$
where $\bar f, \bar a_1, \ldots, \bar a_n, \overline{b^\prime}\in B/\fm_B$. The choice of $x_i$'s implies that $\bar x_1, \ldots, \bar x_n$ is a basis of this vector space. Hence $\bar a_1\overline{b^\prime}=\bar 1$ and $\bar a_2\overline{b^\prime}=\bar f$. So $\bar f=\bar a_1^{-1}.\bar a_2\in A/\fm_A$ and thus $A/\fm_A=B/\fm_B$.
\medskip

\noindent{\it Sufficient condition}: Set $\fn=\fm_B\cap A$. Since $\fm_B=\fn B$, $B\otimes_AA_\fn\simeq B$. So the embedding $A\hookrightarrow B$ factors to two flat homomorphisms $A\rightarrow A_\fn \rightarrow B_\fn=B$. To show that the extension $A\subseteq B$ is Weierstrass, it reduces to showing this for the extension $A_\fn\subseteq B$. Thus without lost of generality, we can assume $A$ is local with a maximal ideal $\fm_A$ and $\fm_AB=\fm_B$. It is worth noting that under this assumption, $\dim A=\dim B$.

If $\fm_B$ is principal then the conclusion is clear. Indeed, since $\fm_AB=\fm_B$, it suffices to show that $\fm_A$ is principal. But since $B/\fm_B^2$ is flat over $A/\fm_A^2$, by Lemma \ref{22}, we have $\dim_{A/\fm_A}\fm_A/\fm_A^2=\dim_{B/\fm_B}\fm_B/\fm_B^2=1$. So $\fm_A$ is principal.

Now assume that $\fm_B$ is not principal. The conclusion is proved by induction on the dimension of the rings. We have $A/\fm_A=B/\fm_B$. So
$$B=A+\fm_B=A+\fm_AB=A+\fm_A+\fm_A^2B=\ldots =A+\fm_A^tB,$$
for all $t>0$. If $\dim A=\dim B=0$ then $\fm_A^t=0$ for some big enough $t\gg 0$. So 
\begin{equation}\label{dimnull}
A=B,
\end{equation}
and the conclusion follows.

Let $\dim A=\dim B>0$. Using Lemma \ref{21}(c), we will show that for any $b\in \fm_B\setminus \{0\}$, $bB=IB$ where $I=(bB)\cap A$. Since base change preserves flatness, the induced homomorphism $A/I\rightarrow B/IB$ is faithfully flat. In particular, it is injective. Let $P\in \Spec(B)$ be a prime ideal such that $I\subseteq P$. We have
$P/IB=\big((P/IB)\cap(A/I)\big)B/IB$, where intersection is taken in the ring $B/IB$. If the ideal $PB_P$ is principal then so is $(P/I)B_P$. If $PB_P$ is not principal then $$k_{B/IB}(P/IB)=k_B(P)=k_A(P\cap A)=k_{A/I}\big((P/IB)\cap (A/I)\big).$$
So the flat embedding $A/I\rightarrow B/IB$ satisfies the sufficient conditions in the theorem. 

There are two cases to  be considered. The first case is $\dim A/I<\dim A$. Then the extension $A/I\rightarrow B/IB$ is Weierstrass by the induction assumption. Using again Lemma \ref{21} we get that $bB/IB=0$ since $(bB/IB)\cap (A/I)=0$. In particular, $bB=IB$. 

The remaining case is $\dim A/I=\dim A$. Replace $A$ by $A/I$ and $B$ by $B/IB$ we can assume without lost of generality that $I=0$. We are going to show that $bB=0$. Let $P\not=\fm_B$ be any prime ideal of $B$ and put $\fp=P\cap A$. As the argument at the beginning of the proof for the necessary condition, $B_P=BA_\fp$ and the induced homomorphism $A_\fp\rightarrow B_P$ satisfies the sufficient condition of the theorem. So $bB_P\cap A_\fp=b(BA_\fp)\cap A_\fp=0$ and thus $bB_P=0$ by the induction assumption, because $\dim B_P<\dim B$. In particular, $bB$ is of finite length. 

We denote by $\Gamma_{\fm_B}(B):=\bigcup_{t>0}(0:_B\fm_B^t)$ the $\fm_B$-torsion ideal of $B$.  The Noetherian assumption of the rings guarantees  $\fm_A^t\Gamma_{\fm_A}(A)=0$ and $\fm_B^t\Gamma_{\fm_B}(B)=0$ for any big enough $t>0$. In particular, we can consider $\Gamma_{\fm_A}(A)$  and $\Gamma_{\fm_B}(B)$ as modules over $A/\fm_A^t$ and $B/\fm_B^t$ respectively. As it has been shown in the case $\dim A=\dim B=0$ (see (\ref{dimnull})), the injective homomorphism $A/\fm_A^t\rightarrow B/\fm_B^t$ is in fact an isomorphism. The flat base change theorem (see \cite[Theorem 4.3.2]{bs}) then gives us isomorphisms
\[\begin{aligned}
\Gamma_{\fm_B}(B)\simeq \Gamma_{\fm_A}(A)\otimes_AB
&\simeq \Gamma_{\fm_A}(A)\otimes_{A/\fm_A^t}A/\fm_A^t\otimes_AB\\
&\simeq \Gamma_{\fm_A}(A)\otimes_{A/\fm_A^t}B/\fm_B^t\\
&\simeq \Gamma_{\fm_A}(A).
\end{aligned}\]
Since the ideal $bB$ is of finite length, we have an inclusion $bB\subseteq \Gamma_{\fm_B}(B)$. The assumption $bB\cap A=0$ implies that $bB=0$ via the isomorphism $\Gamma_{\fm_A}(A)\simeq \Gamma_{\fm_B}(B)$ as above.
\end{proof}

\begin{remark}\label{24} 
Let $\varphi: A\rightarrow B$ be a flat homomorphism of Noetherian rings. For each prime ideal $\fp$ of $A$, the fiber ring of $\varphi$ at $\fp$ is defined as $B\otimes_Ak_A(\fp)$. The Krull dimension of the fiber ring is called the dimension of the fiber at $\fp$ of $\varphi$. It is exactly the dimension of the fiber at the point $\fp$ of the induced morphism $\Spec(B)\rightarrow \Spec(A)$. In condition (a) of Theorem \ref{23}, for a prime ideal $P$ of $B$, $P=\varphi(\fp)B$ for $\fp:=\varphi^{-1}(P)$. This means the fiber of $\varphi$ at $\fp$ is of dimension zero and the ideal $\varphi(\fp)B$ is a prime ideal. On the other hand, if $\fq$ is a prime ideal of $A$, then either $\fq B=B$ or there is a minimal prime ideal $Q$ containing $\fq B$ such that $\varphi^{-1}(Q)=\fq$, as $\varphi$ is flat. So condition (a) is equivalent to 

(a1) All fibers of the flat homomorphism $\varphi$ are of dimension zero.

(a2) For a prime ideal $\fp$ of $A$, either $\fp B=B$ or $\fp B$ is a prime ideal. 

\noindent If this is the case, the fiber ring $B\otimes_Ak_A(\fp)\simeq k_B(\fp B)$ is a field.

\end{remark}

The proof of Theorem \ref{23} immediately implies the following consequence.

\begin{corollary}\label{25}
Let $A\subseteq B$ be a flat extension, where $A$, $B$ are Noetherian local domains. Then the extension $A\subseteq B$ is Weierstrass if and only if for any prime ideal $P\in \Spec(B)$, $P=\fp B$ where $\fp = P\cap A$, and either $B_P$ is a field or a discrete valuation ring or $A_\fp, B_P$ are analytically isomorphic, that is, they have the same completion $\widehat{A_\fp}\simeq\widehat{B_P}$. 
\end{corollary}
\begin{proof} Using Theorem \ref{23} it suffices to show that if $A\subseteq B$ is a Weierstrass extension and $PB_P$ is not principal then $A_\fp$ and $B_P$ are analytically isomorphic. Replace the extension $A\subseteq B$ by $A_\fp\subseteq B_P$ we can assume without lost of generality that $P=\fm_B$, $\fp=\fm_A$ are maximal ideals of $B$ and $A$ respectively. Since $A\rightarrow B$ is faithfully flat with $\fm_B=\fm_AB$, the same argument as in the proof of Theorem \ref{23} (the sufficient condition part) shows that $B/\fm_B^t\simeq A/\fm_A^t$ for any $t>0$. This proves the corollary.
\end{proof}




\begin{corollary}\label{26}
Let $\varphi: A\rightarrow B$ be a flat homomorphism of Noetherian local rings. Suppose the homomorphism is Weierstrass. Then it satisfies both the going-up and going-down theorems.
\end{corollary}
\begin{proof}
The flatness of $\varphi$ implies that it satisfies the going-down theorem by \cite[Theorem 9.5]{mat1}. For the going-up theorem, let $\fp\subset \fq$ be two prime ideals of $A$ and let $P$ be a prime ideal of  $B$ such that $P\cap A=\fp$. We need to show that there is a prime ideal $Q$ of $B$ such that $P\subset Q$ and $Q\cap A=\fq$. 

Put $Q=\fq B$. Since the homorphism $\varphi$ is Weierstrass, $Q$ is a prime ideal of $B$ as we have seen in Remark \ref{24}. Similarly, $P=\fp B$. In other words, $P$ and $Q$ are unique point in the fibers of $\varphi$ at $\fp$ and $\fq$ respectively. Then $Q\supset P$ by the going-down theorem. The conclusion follows trivially.
\end{proof}

We end this section with a simple remark that will be used in several places later on.
\begin{remark} \label{27}
Let $\varphi: A\rightarrow B$ be a flat homomorphism of Noetherian rings. The ring $B$ might not be local but we assume that $\varphi$ satisfies the sufficient condition in Theorem \ref{23}, that means, any prime ideal $P\in \Spec(B)$ satisfies

(a) $P=\varphi(\varphi^{-1}(P))B$;

(b) Either $B_P$ is a principal ideal ring or the canonical inclusion $k_A(\varphi^{-1}(P))\hookrightarrow k_B(P)$ is an isomorphism.

For any prime ideal $Q$ of $B$, denote the composition homomorphism of $\varphi$ and the localization map $B\rightarrow B_Q$ by $\varphi_Q: A\rightarrow B_Q$. Then $\varphi_Q$ is flat. It is easy to see that $\varphi_Q$ satisfies similar conditions as (a) and (b). So by Theorem \ref{23}, $\varphi_Q$ is Weierstrass.
\end{remark}


\section{Functions on curves}
\label{sect3}

In this section we consider rings of functions on a projective curve over a complete local domain of dimension one. Those curves contain all curves over a complete discrete valuation ring. Using Theorem \ref{23} we prove some interesting Weierstrass extensions of these rings. It generalizes the results of Harbater-Hartmann \cite{hh} and Harbater-Hartmann-Krashen in \cite{hhk1, hhk2} which in fact motivate our work in this paper.

\begin{notation}\label{31} To fix the notations, let $(R,\fm_R)$ be a complete Noetherian local domain of dimension one. We consider a connected projective normal curve $\mathcal X$ over $R$ with closed fiber $X$ and function field $F$. Assume $\mathcal X$ dominates $\Spec(R)$. Given a subset  $U\subset X$ which is contained in an affine set, let $R_U$ be the set of rational functions in $F$ which are regular on $U$. We denote by $\widehat R_U$ the $\fm_R R_U$-adic completion of the ring $R_U$.
\end{notation}

By the choice of the subset $U$, the ideal $\fm_R R_U$ is contained in the Jacobson radical of $R_U$ and the canonical homomorphism $R_U\rightarrow \widehat R_U$ is faithfully flat.

We will explore some Weierstrass extensions of $R_U$ in the following three situations

(i) $U$ consists of a closed point of $X$;

(ii) $U$ is an open subset of an irreducible component of $X$ and $U$ does not intersect with other components;

(iii) $U$ is general.

The first two situations have been worked out by Harbater, Hartmann and Krashen using the techniques of the patching theory (see \cite{hh}, \cite{hhk1}, \cite{hhk2}). Also they work on the field of fractions of $R_U$ rather than on $R_U$. Here we work on the level of rings. The main idea is to use Theorem \ref{23}, which in fact enables us to study in addition the third case of a general subset $U$ too.

Let $\varphi: A \rightarrow B$ be a Weierstrass flat homomorphism. We have seen in Lemma \ref{21} and Theorem \ref{23} that if $P$ is a prime ideal of $B$ then $P=\varphi(\varphi^{-1}(P))B$. We study first this property for the ring $R_U$.

\begin{proposition}\label{32} 
Keep the notations in \ref{31}. For any prime ideal $\fp$ of $R_U$, $\fp\not=0$, the canonical embedding $R_U/\fp \hookrightarrow \widehat{R}_U/\fp\widehat R_U$ is an isomorphism. In particular, $\fp\widehat R_U$ is a prime ideal of $\widehat R_U$. Conversely, for a prime ideal $Q$ of $\widehat R_U$ which is not the generic point of any irreducible component of $\Spec(\widehat R_U)$, denote $\fq=Q\cap R_U$, then $Q=\fq\widehat R_U$ and $k_{R_U}(\fq)\simeq k_{\widehat{R}_U}(Q)$.
\end{proposition}

\begin{proof} For the first assertion, let $\fp\not=0$ be a prime ideal of $R_U$. By the choice of $U$, the ideal $\fp+\fm_RR_U$ is contained in some maximal ideal of $R_U$. If $\fp\supseteq \fm_R$ then $R_U/\fp= \widehat{R}_U/\fp \widehat{R}_U$. It particularly holds true for any maximal ideal of $R_U$. 

Suppose $\fp\not\supseteq \fm_R$. Since $R_U$ has Krull dimension $2$ and $\fp\not=0$, $R_U/(\fp+\fm_RR_U)$ is Artinian. Hence $\widehat R_U/(\fp\widehat R_U+\fm_R\widehat R_U)\simeq R_U/(\fp+\fm_RR_U)$ is also Artinian. In other words, $\sqrt{\fp\widehat R_U+\fm_R\widehat{R}_U}=\hat \fm_1\cap\ldots\cap \hat \fm_h$ for some maximal ideals $\hat \fm_1, \ldots, \hat \fm_h$ of $\widehat R_U$. Let $\fm_i=R_U\cap \hat \fm_i$, $i=1, \ldots, h$. Since $\fm_i\supseteq \fm_R$, $R_U/\fm_i=\widehat{R}_U/\fm_i\widehat{R}_U$. This shows that $\hat \fm_i=\fm_i\widehat{R}_U$ and $\fm_i$ is a maximal ideal of $R_U$, $i=1, \ldots, h$. From the equality $\sqrt{\fp\widehat R_U+\fm_R\widehat{R}_U}=\fm_1\widehat R_U\cap\ldots\cap \fm_h\widehat R_U=(\fm_1\cap \ldots \cap \fm_h)\widehat R_U$, there is an $n>0$ such that
$$(\fm_1\cap \ldots \cap \fm_h)^{n+1}\widehat R_U=(\fm_1\widehat R_U\cap \ldots\cap \fm_h\widehat R_U)^{n+1}\subseteq \fp\widehat R_U+\fm_R\widehat R_U.$$
On the closed fiber $X$ of the curve $\cX$, the residue field $\widehat R_U/\widehat \fm_i\simeq R_U/\fm_i$ is a finite extension of $k:=R/\fm_R$ by the geometric formulation of Hilbert's Nullstellensatz (see \cite[Proposition 3, page 99]{Mum}). So we can write 
$$\widehat R_U/\widehat \fm_i=k\bar w_{i1}+\ldots+k\bar w_{it},$$
for some $w_{i1}, \ldots, w_{it}\in R_U$, $t>0$. The isomorphisms
$$\widehat R_U/(\fm_1\cap\ldots\cap \fm_h)\widehat R_U
\simeq \widehat R_U/(\widehat \fm_1\cap \ldots \cap \widehat \fm_h)
\simeq \widehat R_U/\widehat \fm_1\oplus \ldots \oplus \widehat R_U/\widehat \fm_h,$$
give us
\[\begin{aligned}
\widehat R_U&=Rw_{11}+\ldots+Rw_{1t}+\ldots +Rw_{h1}+\ldots+Rw_{ht}+(\fm_1\cap \ldots \cap\fm_h)\widehat R_U\\
&=R[w_{ij}] + (f_1, \ldots, f_r)\widehat R_U,
\end{aligned}\]
 where $R[w_{ij}]$ denotes the $R$-module generated by $w_{11}, \ldots, w_{ht}$ and $f_1, \ldots, f_r$ are rational functions on $\cX$ that generate $\fm_1\cap\ldots\cap \fm_h$. Denote
$$R[w_{ij}][f_1, \ldots, f_r]_{\leq s}:=\sum_{\substack{1\leq i\leq h,\, 1\leq j\leq t\\ \alpha_1+\ldots+\alpha_r\leq s}}Rw_{ij}f_1^{\alpha_1}\ldots f_r^{\alpha_r}.$$
Then 
\[\begin{aligned}
\widehat R_U&=R[w_{ij}][f_1, \ldots, f_r]_{\leq 1} + (f_1,\ldots, f_r)^2\widehat R_U\\
&=\ldots\\
&=R[w_{ij}][f_1, \ldots, f_r]_{\leq n} + (f_1,\ldots, f_r)^{n+1}\widehat R_U\\
&=R[w_{ij}][f_1, \ldots, f_r]_{\leq n} + (\fp\widehat R_U+\fm_R\widehat R_U).
\end{aligned}\]
Let's denote $M=\widehat R_U/\fp\widehat R_U$. We have $M=R[w_{ij}][\bar f_1, \ldots, \bar f_r]_{\leq n}+\fm_RM$. Consequently, the $k$-vector space $M/\fm_RM$ is generated by $\bar w_{ij}\bar f_1^{\alpha_1}\ldots \bar f_r^{\alpha_r}$, for $1\leq i\leq h, 1\leq j\leq t, \alpha_1+\ldots+\alpha_r\leq n$. Since $R$ is $\fm_R$-adically complete, $M$ is separated with respect to the $\fm_RR_U$-adic topology. A theorem of Cohen (see \cite[Theorem 30.6]{nag} or \cite[Theorem 8.4]{mat1}) concludes that the set of generators $\bar w_{ij}\bar f_1^{\alpha_1}\ldots \bar f_r^{\alpha_r}$ could be lifted over $R$, that is, $M=R[w_{ij}][\bar f_1, \ldots, \bar f_r]_{\leq n}$. Hence
$$M\subseteq R_U/\fp\subseteq \widehat R_U/\fp\widehat R_U=M,$$
and therefore $R_U/\fp=\widehat R_U/\fp\widehat R_U$. In particular, $\fp\widehat R_U$ is a prime ideal of $\widehat R_U$.

For the converse, let $Q$ be a prime ideal of $\widehat R_U$ which is not minimal. Put $\fq=Q\cap R_U$. By the same proof as the first part in which we replace $\fp\widehat R_U$ by $Q$, we get that $R_U/\fq\simeq \widehat R_U/Q$. In particular, $\fq\not=0$. As we have proved above, $R_U/\fq\simeq \widehat R_U/\fq\widehat R_U$. So $Q=\fq\widehat R_U$ and $k_{R_U}(\fq)\simeq  k_{\widehat R_U}(Q)$.
\end{proof}

Proposition \ref{32} provides us a nice 1-1 correspondence between the sets of non-minimal prime ideals of $R_U$ and $\widehat R_U$. As a particular consequence, it induces that all fibers of the completion homomorphism $R_U\rightarrow \widehat R_U$ have dimension zero. This description shows part of the sufficient condition in Theorem \ref{23} (see Remark \ref{24}, condition (a1)). Using this we are able to give another proof for one of the main theorems of \cite{hhk2} (see also \cite{hh}, \cite{hhk1}) and also to study the general case where the set $U$ contains points on different irreducible components of the closed fiber of the curve.

\begin{theorem}{(See also \cite[Theorem 3.1]{hhk2})}\label{33}
Keep the notations in \ref{31}. Let $U$ be either 

(a) A set consisting of one closed point of the closed fiber $X$; or

(b) An open subset of an irreducible component of $X$ which does not intersect other components.

\noindent For each prime ideal $Q$ of $\widehat R_U$, the canonical inclusion $R_U\hookrightarrow (\widehat R_U)_Q$ is Weierstrass.
\end{theorem}
\begin{proof}
We first prove that $\widehat R_U$ is a domain. \smallskip

\noindent (a) Let $U=\{ x\}$ where $x$ is a closed point of the closed fiber $X$ of the curve $\cX$. We will write $R_x$ instead of $R_U$. Let $\widetilde{R}_x$ be the completion of the ring $R_x$ with respect to its maximal ideal. We can write $R_x\subset \widehat R_x \subset \widetilde{R}_x$ as faithfully flat extensions. The ring $R_x$ is excellent since it is essentially of finite type over the complete local ring $R$, following \cite[Scholie 7.8.3]{ega1965}. It implies in particular that $\widetilde{R}_x$ is reduced. On the other hand, $R_x$ is a normal ring since $\cX$ is normal, thus $R_x$ is particularly unibranch. The ring $\widetilde{R}_x$ is then a domain as it is both reduced and unibranch. So $\widehat{R}_x$ is a domain. This proves the assertion in the first case.\smallskip

\noindent (b) Now let $U$ be an open subset of a component of $X$ which does not intersect other components. Let $x\in U$ be a closed point. The inclusion $R_U/\fm_R^n R_U\hookrightarrow R_x/\fm_R^n R_x$ gives rise to an inclusion $\widehat R_U\hookrightarrow \widehat R_x$. So the assertion follows from the first case.

Combining this with Proposition \ref{32}, we see that the canonical inclusion $R_U\subset \widehat R_U$ satisfies

(i) For any prime ideal $P$ of $\widehat R_U$, $P=(P\cap R_U)\widehat R_U$;

(ii) For any non-zero prime ideal $P$ of $\widehat R_U$, $k_{\widehat R_U}(P)=k_{R_U}(P\cap R_U)$.

\noindent Therefore, with a prime ideal $Q$ of the ring $\widehat R_U$, Theorem \ref{23} and Remark \ref{27} imply that the extension $R_U \subset (\widehat R_U)_Q$ is Weierstrass.
\end{proof}

\begin{remark}\label{34} 
In the proof of Theorem \ref{33} it is shown that the canonical inclusion $R_U\subset \widehat R_U$ fulfills all the conditions in Theorem \ref{23} except $\widehat R_U$ is not local in general. So Theorem \ref{23} is not applicable for this extension (see also the discussion in Remark \ref{27}). In fact, we do not know whether the extension $R_U\subset \widehat R_U$ itself is Weierstrass or not.
\end{remark}

In Theorem \ref{33}, if $U=\{x\}$ consists of just one closed point, the ring $R_x$ is local with a maximal ideal $\fm_x$. Beside the extension $R_x\subset \hat R_x$, there is another extension $R_x\subset \widetilde R_x$ - the $\fm_x$-adic completion. It is natural to ask whether this extension is Weierstrass. The answer is positive with some mild restrictions.

\begin{theorem}\label{35}
Keep the notations in \ref{31}. Let $x$ be a closed point of the closed fiber $X$ of the curve $\cX$. For each irreducible component $X_i$ of $X$ which contains $x$, assume that the ring $\cO_{X_i, x}$ is analytically unramified and unibranch, or equivalently, its  $\fm_x$-adic completion is a domain (this is the case if $x$ is a non-singular point of $X_i$). Then the $\fm_x$-adic completion inclusion $R_x\subset \widetilde R_x$ is Weierstrass.
\end{theorem}
\begin{proof}
We consider the extensions $R_x\subset \widehat R_x\subset \widetilde R_x$. The first inclusion is Weierstrass by Theorem \ref{33}. We need to show that the second inclusion is also Weierstrass. Note that the ring $\widetilde R_x$ is the $\fm_x\widehat R_x$-adic completion of $\widehat R_x$. 

We first show that if $\fp$ is a prime ideal of $\widehat R_x$, $\fp \widetilde R_x$ is a prime ideal. If either $\fp=\widehat\fm_x$ is the maximal ideal of $\widehat R_x$ or $\fp=0$ then this is clear. We assume $\fp\not= \widehat \fm_x$ and $\fp\not=0$. There are two cases to be considered separately. The first case is if $\fp\not\supseteq \fm_R$. The ring $\widehat R_x/\fm_R\widehat R_x$ is of Krull dimension one, thus  $\sqrt{\fm_R\widehat R_x+\fp}=\widehat \fm_x$. We have
$$\widetilde R_x/\fp \widetilde R_x
\simeq \widetilde{(\widehat R_x/\fp)}
\simeq \varprojlim_{n} \widehat R_x/\widehat\fm_x^n+\fp
\simeq \varprojlim_{n} \widehat R_x/(\fm_R\widehat R_x+\fp)^n+\fp.$$
Hence
$$\widetilde R_x/\fp \widetilde R_x
\simeq \varprojlim_{n} \widehat R_x/\fm_R^n\widehat R_x+\fp
\simeq \widehat R_x/\fp,$$
as $\widehat R_x/\fp$ is complete with respect to the $\fm_R\widehat R_x$-adic topology. 
So $\fp\widetilde R_x$ is a prime ideal. The second case is if $\fp\supseteq \fm_R$. Let $X_i$ be the irreducible component of the closed fiber $X$ associated to $\fp\cap R_x$. We have the embedding
$$\cO_{X_i, x}=R_x/\fp\cap R_x\simeq \widehat R_x/\fp \hookrightarrow \widetilde R_x/\fp\widetilde R_x \simeq \widetilde{\cO_{X_i, x}}.$$
The completion $\widetilde{\cO_{X_i, x}}$ being a domain implies that $\fp \widetilde R_x$ is a prime ideal.

We now prove that all the fibers of the extension $\widehat R_x\subseteq \widetilde R_x$ have dimension zero. Let $P$ be a prime ideal of $\widetilde R_x$. We need to prove that if $P$ is neither minimal nor maximal, then so is the prime ideal $\fp:=P\cap \widehat R_x$  of the ring $\widehat R_x$. If $P\supseteq \fm_R$ then $\fp\supseteq \fm_R$ which is obviously not a minimal prime ideal. Assume that $P\not\supseteq \fm_R$. So $P\not\supseteq \widehat \fm_x$. Since the ring $\widetilde R_x/P$ is of Krull dimension one, $\sqrt{\widehat \fm_x\widetilde R_x+P}=\widetilde \fm_x$. A similar argument as in the proof of Proposition \ref{32} concludes that $\widetilde R_x/P$ is finitely generated over $\widehat R_x/\fp$. Hence $\dim \widetilde R_x/P=\dim \widehat R_x/\fp$, and thus $\fp$ is not a minimal prime ideal of the ring $\widehat R_x$.

Combining all the fact above, we get that any prime ideal $P$ of $\widetilde R_x$ is generated by $P\cap \widehat R_x$, following Remark \ref{24}. Moreover, if $P=\fm_x\widetilde R_x$ then clearly $k_{\widetilde R_x}(P)=k_{\widehat R_x}(\widehat \fm_x)=k_{R_x}(\fm_x)$. If $P$ is not maximal then as $\cX$ is normal and excellent, $\widetilde R_x$ is normal. In particular, $(\widetilde R_x)_P$ is a discrete valuation ring. The extension $\widehat R_x\subseteq \widetilde R_x$ is therefore Weierstrass by Theorem \ref{23}.
\end{proof}

We now consider the case of general subset $U$ which might contain points on different irreducible components of the closed fiber of the curve $\cX$. In this case, $R_U$ is a domain but the $\fm_RR_U$-adic completion $\widehat R_U$ could be no more a domain. However, $\widehat R_U$ is still reduced and equidimensional as $R_U$ is an excellent domain. Let $P_1, \ldots, P_t$ be the minimal prime ideals of $\widehat R_U$. Taking the composition of two canonical maps $R_U\subset \widehat R_U$ and $\widehat R_U \rightarrow \widehat R_U/P_i$, we get a canonical inclusion $R_U\hookrightarrow \widehat R_U/P_i$ which is not flat in general. However, we still have the following theorem.

\begin{theorem}\label{36}
Keep the notations in \ref{31}. Let $U$ be a subset of the closed fiber of the normal curve $\cX$. Let $\widehat R_U$ be the $\fm_RR_U$-adic completion of $R_U$ and $P_1, \ldots, P_t\in \Spec \widehat R_U$ be generic points of the irreducible components of $\Spec \widehat R_U$. Let $Q$ be a maximal ideal of $\widehat R_U$. For each $P_i\subseteq Q$, the canonical inclusion $R_U\hookrightarrow (\widehat R_U/P_i)_Q$ is Weierstrass.
\end{theorem}
\begin{proof}

In order to show the Weierstrass property of the canonical homomorphism $\varphi: R_U \rightarrow (\widehat R_U/P_i)_Q$, by using Lemma \ref{21}(c), we will show that any principal ideal $J$ of $(\widehat R_U/P_i)_Q$ is generated by a set of generators in $R_U$ via $\varphi$. Suppose $J\not=0$ and $J\not=(\widehat R_U/P_i)_Q$. Let $J=\big(\frac{\bar b}{1}\big)$ for some $b\in Q$ and $I:=\varphi^{-1}(J)$.
Clearly, $Q\cap R_U\supseteq I\supseteq (b\widehat R_U+P_i)\cap R_U$. It is worth noting that $b\widehat R_U+P_i\not\subseteq \cup_{j=1}^t P_j$. So $I\not=0$ since  $P_1, \ldots, P_t$ are the only prime ideals in the generic fiber of of the completion map $R_U\rightarrow \widehat R_U$ (cf. Proposition \ref{32}).

The homomorphism $\varphi$ induces a homomorphism $\bar\varphi: R_U/I\rightarrow (\widehat R_U/I\widehat R_U+P_i)_Q$ which factors as
$$R_U/I\rightarrow (\widehat R_U/I\widehat R_U)_Q\rightarrow (\widehat R_U/I\widehat R_U+P_i)_Q.$$
Note that $(I\widehat R_U+P_i/P_i)_Q\subseteq J$ and we will denote the quotient by $\bar J$. Then $\bar J=\big(\frac{\bar b}{1})$ and $\bar \varphi^{-1}(\bar J)=0$. 

On the other hand, since $I\not=0$, the extension 
$$R_U/I \rightarrow \widehat{(R_U/I)}_Q\simeq (\widehat R_U/I\widehat R_U)_Q,$$
is Weierstrass following Proposition \ref{32} and Remark \ref{27}. As a consequence, the induced homomorphism
$$R_U/I \rightarrow (\widehat R_U/I\widehat R_U+P_i)_Q,$$
is Weierstrass. Hence $\bar J=0$ as it is generated by $\bar \varphi(\bar\varphi^{-1}(\bar J))=0$. It implies that in the ring $(\widehat R_U/P_i)_Q$, $J=I(\widehat R_U/P_i)_Q$ as required.
\end{proof}

Theorem \ref{36} provides a Weierstrass extension $R_U\rightarrow (\widehat R_U/P_i)_Q$ where $Q, P_i$ are as in the theorem. The ring $(\widehat R_U/P_i)_Q$ is henselian with respect to the ideal $(\fm_R)$. Recall that $F$ is the function field of the normal curve $\cX$. We denote in addition by $\widehat F_{U, P_i}$ the field of fractions of $\widehat R_U/P_i$. If the residue field $k$ of $R$ has characteristic different from $2$ then Theorem \ref{15} applies and we have $u(\widehat F_{U, P_i})\leq u(F)$. We even have more. Following Harbater-Hartmann-Krashen, the strong u-invariant $u_s(k)$ of a field $k$ is defined as the smallest integer $n$ such that $u(k^\prime)\leq n$ for any finite field extension $k^\prime/k$ and $u(L)\leq 2n$ for any finitely generated field extension of transcendence degree one $L/k$.

\begin{corollary}\label{37}
Keep the notations in \ref{31} and the hypothesis of Theorem \ref{36}. Suppose the residue field $k$ of $R$ has characteristic different from $2$. Then 
$$u(\widehat F_{U, P_i})\leq 4u_s(k).$$
\end{corollary}
\begin{proof} 
We have seen $u(\widehat F_{U, P_i})\leq u(F)$. So clearly $u(\widehat F_{U, P_i})\leq 2u_s(K)$, where $K$ is the field of fractions of $R$. Let $R^\prime$ be the integral closure of $R$ in $K$. Since $R$ is henselian, $R^\prime$ is local and finite over $R$. So $R^\prime $  is a complete discrete valuation ring. Due to Harbater-Hartmann-Krashen \cite[Theorem 4.10]{hhk1}, $u_s(K)=2u_s(k^\prime)$ where $k^\prime$ is the residue field of $R^\prime$. Note that $k^\prime$ is a finite extension of $k$. So $u_s(k^\prime)\leq u_s(k)$ by the definition. Therefore, $u(\widehat F_{U, P_i})\leq 4u_s(k)$.
\end{proof}


\section{Henselization and completion}
\label{sect4}

In this section we will study the Weierstrass property of the Henselization and the completion of a local ring. In applications of the Weierstrass preparation theorem, we are interested in those Weierstrass extensions $A\subseteq B$ such that $B$ is henselian (see Theorem \ref{15}). While the Weierstrass property indicates that the two rings $A, B$ are very close, such an extension probably does not exist because of the henselian requirement on the ring $B$. In seeking such extensions, a natural idea is to look at the Henselization and the completion  of the ring. Later we will see that, in many cases the Henselization and the completion are in fact two bounds, lower and upper. The next proposition illustrates part of this point.

\begin{proposition}\label{41}
Let $(A, \fm_A)\subseteq (B, \fm_B)$ be a faithfully flat ring extension consisting of Noetherian local rings. Assume that the extension is Weierstrass and the ring $B$ is henselian. Denote the Henselization of $A$ by $A^h$. Then we have faithfully flat embeddings $A\hookrightarrow A^h \hookrightarrow B$ whose composition is the inclusion $A\subseteq B$. If $B$ is not a principal ideal ring then we have the faithfully flat embeddings $A^h \hookrightarrow B\hookrightarrow \hat A$, where $\hat A$ is the $\fm_A$-adic completion of $A$. 

Moreover, the Henselization $A\hookrightarrow A^h$ is Weierstrass.
\end{proposition}
\begin{proof}
The universality of Henselization gives rise to the existence of a homomorphism $\psi: A^h\rightarrow B$ making the following triangle commutative

\centerline{
\xymatrix@1{
A\ \ \ar@{^{(}->}[rr]
    \ar@{_{(}->}[dr] && \ A^h\ar[dl]^\psi.\\
&B}
}

\noindent This induces a commutative diagram

\centerline{
\xymatrix@1{
A/\fm_A^n\ \ar@{^{(}->}[rr]
    \ar@{_{(}->}[dr]\  &&\ A^h/\fm_A^nA^h\ar[dl]^{\bar \psi}\ \\
&B/\fm_A^nB}
}

\noindent where $A^h/\fm_A^nA^h\simeq (A/\fm_A^n)^h$. Taking the limit we get a commutative diagram

\centerline{
\xymatrix@1{
\hat A\ \ar[rr]^\simeq
    \ar@{_{(}->}[dr]\  &&\ \widehat{A^h}\ar[dl]^{\hat \psi}\ \\
&\hat B}
}

\noindent Note that the map $\hat A \rightarrow \widehat{A^h}$ is in fact an isomorphism (cf. \cite[Th\'eor\`eme 18.6.6]{ega1967}). Hence $\hat\psi$ is injective. Identify $A^h$ with a subring of $\widehat{A^h}$, we get that $\Ker(\psi)\subseteq \Ker(\hat \psi)=0$. So $\psi$ is injective.

If $B$ is not a principal ideal ring, then $\hat A\simeq \hat B$ following Corollary \ref{25}. So there is a faithfully flat embedding $B\hookrightarrow \hat A$. 

In order to show that the extension $A\hookrightarrow A^h$ is Weierstrass, we note that the induced flat homomorphism $A/\fp \hookrightarrow B/\fp B$ is Weierstrass for any prime ideal $\fp$ of $A$. By the first part of the proof we obtain injective homomorphisms $A/\fp\hookrightarrow (A/\fp)^h\hookrightarrow B/\fp B$. Note also that $(A/\fp)^h\simeq A^h/\fp A^h$. So $A^h/\fp A^h$ is a domain, or equivalently, $\fp A^h$ is a prime ideal of $A^h$. Combining this with the fact that the Henselization $A \subseteq A^h$ has all fibers of dimension zero (cf. \cite[Th\'eor\`eme 18.6.9]{ega1967}), we conclude that $P=(P\cap A)A^h$ for any prime ideal $P$ of $A^h$ by using Remark \ref{24}. Moreover, from the argument above we have the inclusions
$$k_A(\fp)\subseteq k_{A^h}(\fp A^h)\subseteq k_B(\fp B).$$
If the ideal $\fp B$ is not principal, then $k_A(\fp)=k_{A^h}(\fp A^h)=k_B(\fp B)$ by Theorem \ref{23}. Therefore, the Henselization $A\subseteq A^h$ is Weierstrass by Theorem \ref{23} again.
\end{proof}

Taking the images, one can assume $A\subseteq A^h\subseteq B$. Suppose in addition that the characteristic of the residue fields is either zero or an odd prime, then by Theorem \ref{15} one get a comparison of u-invariants $u(K_A)\geq u(K_{A^h})\geq u(K_B)$.

We have seen in Proposition \ref{41} that the Henselization $A\subseteq A^h$ is Weierstrass provided the existence of a henselian Weierstrass extension. The conclusion in fact holds true with a much weaker hypothesis. Recall that a local ring $R$ is unibranch if the reduction $R_\mathrm{red}$ is a domain and the integral closure of $R$ in its field of fractions is a local ring.

\begin{proposition}\label{42}
Let $A$ be a Noetherian local domain. The Henselization $A\subseteq A^h$ is Weierstrass if and only if
\begin{enumerate}
\item[(a)] For any prime ideal $\fp\in \Spec(A)$, $A/\fp$ is unibranch;
\item[(b)] For any prime ideal $\fp\in \Spec(A)$, if $\fp\not=0$ and $A_\fp$ is not a discrete valuation ring then $k_A(\fp)\simeq k_{A^h}(\fp A^h)$.
\end{enumerate}
\end{proposition}
\begin{proof}
We know that the fibers of the Henselization $A\rightarrow A^h$ are regular of dimension zero (see \cite[Th\'eor\`eme 18.6.9]{ega1967}). So by Theorem \ref{23} and Remark \ref{24}, $A\rightarrow A^h$ is Weierstrass if and only if it satisfies the condition (b) in Theorem \ref{23} and the following condition $(a')$: for any prime ideal $\fp\in \Spec(A)$, $\fp A^h$ is a prime ideal of $A^h$. Let $(A/\fp)^\prime$ be the integral closure of $A/\fp$ in its field of fractions. A beautiful theorem of Nagata \cite[Theorem 43.20]{nag} says that there is a $1-1$ correspondence between the maximal ideals of $(A/\fp)^\prime$ and the associated prime ideals of $(A/\fp)^h$. The isomorphism $A^h/\fp A^h\simeq (A/\fp)^h$ for any prime ideal $\fp$ of $A$ implies that $A^h/\fp A^h$ is reduced (again, because the fibers of the Henselization are regular). So $A^h/\fp A^h$ is a domain if and only if $A/\fp$ is unibranch. This proves the equivalence of $(a)$ and $(a')$.
\end{proof}

As an application of Proposition \ref{42}, we get the following immediate consequence.
\begin{corollary}\label{43}
Let $A$ be a Noetherian local ring of dimension $2$ which is regular in codimension one (for example, if $A$ is normal). Suppose for any prime ideal $\fp\in \Spec(A)$, $A/\fp$ is unibranch. The Henselization $A\rightarrow A^h$ is Weierstrass.
\end{corollary}

The Henselization serves as a lower bound for Weierstrass faithfully flat extensions of a local ring as we have seen in Proposition \ref{41}. We now consider an upper bound - the completion. Let $(A, \fm_A)\subseteq (B, \fm_B)$ be a Weierstrass faithfully flat extension. We have $A^h\hookrightarrow  B$ and $\hat A\simeq \widehat{A^h}\hookrightarrow \hat B$. If $B$ is a principal ideal ring then the inclusion $\hat A\rightarrow \hat B$ is not an isomorphism in general.

\begin{example}\label{44}
Let $A=\bR[[X]]$ and $B=\bC[[X]]$ be the rings of formal power series over the fields of real and complex numbers. The natural embedding $\bR[[X]]\rightarrow \bC[[X]]$ is faithfully flat and Weierstrass which is not an isomorphism.
\end{example}

If $B$ is not a principal ideal ring then the map $\hat A\hookrightarrow \hat B$ is an isomorphism. We obtain the faithfully flat embeddings
$$A^h\hookrightarrow B\hookrightarrow \hat A.$$
There comes a natural question whether the extension $A\subseteq \hat A$ is Weierstrass.

\begin{corollary}\label{45}
Let $A$ be a Noetherian local domain of dimension one. The completion $A\subseteq \hat A$ is Weierstrass if and only if $A$ is analytically unramified and unibranch.
\end{corollary}
\begin{proof}
Since $A$ is a local domain of dimension one, Theorem \ref{23} shows that the extension $A\subseteq \hat A$ is Weierstrass if and only if $\hat A$ is a domain. This is equivalent exactly to the sufficient condition in the corollary.
\end{proof}
In dimension two, Theorem \ref{35} brings to us interesting examples from functions on curves over a complete local domain of dimension one. For higher dimension, examples of local rings with Weierstrass completion are in fact quite rare.

\begin{example}\label{46}
Let $X$ be a projective variety over a field $k$ and $x$ be a point of $X$. The dimension of the generic fiber of the completion map $\cO_{X, x}\rightarrow \widehat \cO_{X, x}$ is $\codim_x(X)-1$ (see the proof of Proposition \ref{47} below). So if the completion map is Weierstrass, we get $\codim_x(X)=1$ following Remark \ref{24}. If $x$ is a closed point, this simply means $X$ being a curve.
\end{example}

Though in this example the flat extension $\cO_{X, x}\subseteq \widetilde \cO_{X, x}$ is often not Weierstrass, there are Weierstrass non-flat extensions induced from it. The example suggests that the following proposition could apply to many local rings coming from geometry.

\begin{proposition}\label{47}
Let $(A, \fm_A)$ be a Noetherian local ring. Let $P$ be a prime ideal of the $\fm_A$-adic completion $\hat A$. We denote by $\bar A$ the image of $A$ through the composition of maps $A\hookrightarrow \hat A\rightarrow \hat A/P$. The ideal  $P$ is maximal in the generic fiber of the completion map $A\hookrightarrow \hat A$ if and only if for each $\bar a\in \hat A/P$, there is an element $0\not=\bar b \in \hat A/P$ such that $\bar a\bar b \in \bar A$.

If in addition $\dim \hat A/P=1$, then $\bar b$ could be chosen to be invertible. Equivalently, the induced homomorphism $A\rightarrow \hat A/P$ is Weierstrass.
\end{proposition}
\begin{proof}

Replacing $A$ by $A/P\cap A$, we can assume that $P\cap A=0$. 

Suppose $P$ is maximal in the generic formal fiber of $A$. Take $\bar a\in \hat A/P$, $\bar a\not=0$. If we denote $S=A\setminus\{0\}$ then clearly $S^{-1}A\rightarrow S^{-1}(\hat A/P)$ is a field extension (usually transcendental). In particular, there is an element $\frac{\bar{a^\prime}}{s}\in S^{-1}(\hat A/P)$ where $\bar{a^\prime}\in \hat A/P$, $s\in S$, such that $\frac{\bar a.\bar{a^\prime}}{s}=1$. The conclusion then follows.

Conversely, it is obvious that $S^{-1}(\hat A/P)$ is a field. This shows the maximality of $P$ in the generic formal fiber of $A$.

Now we assume $\dim \hat A/P=1$. Take an element $a\in \widehat \fm_A$ such that $a\not\in P$. Then $\sqrt{a\hat A+P}=\widehat\fm_A$. Let $I=(a\hat A+P)\cap A$. The ideal $I$ is $\fm_A$-primary, thus the canonical homomorphism $A/I\rightarrow \widehat{A/I}\simeq \hat A/I\hat A$ is an isomorphism. This particularly implies $(a\hat A+P)/I\hat A=0$ as its restriction to $A/I$ is zero. Hence $I\hat A=a\hat A+P$. Let $x_1, \ldots, x_n$ be a set of generators of the ideal $I$. We write $a=\sum_{i=1}^n\alpha_ix_i$ for some $\alpha_i\in \hat A$. In the ring $\hat A/P$ we also have $\overline{x_i}=\beta_i\bar a$ for some $\beta_i\in \hat A$. Then 
$$\bar a(1-\sum_{i=1}^n\alpha_i\beta_i)=0.$$
As $\bar A/P$ is a domain and $\bar a\not=0$, the above equality shows that $1-\sum_{i=1}^n\alpha_i\beta_i\in P$. In particular, there is an invertible element $\beta_i$ in $\hat A$. So $\bar a=\beta_i^{-1}.\overline{x_i}$, where $x_i\in I\subset A$.
\end{proof}

\medskip
\noindent{\bf Acknowledgments.} The author is grateful to Julia Hartmann for helpful comments on the first draft of the manuscript and to Kay R\"ulling for interesting discussion on the ring of functions on curves in Section \ref{sect3}. We would like to thank the referees for their detailed suggestions for improvement to this paper. This work is finished during the author's postdoctoral fellowship at the Vietnam Institute for Advanced Study in Mathematics (VIASM). He thanks VIASM for financial support and hospitality.


\end{document}